\documentclass[12pt]{amsart}%
\usepackage{amscd,amsmath,latexsym,amsthm,amsfonts,amssymb,graphicx,color,geometry}
\usepackage{enumerate}
\usepackage[pdftex,plainpages=false,pdfpagelabels,
colorlinks=true,citecolor=blue,linkcolor=black,urlcolor=black,
filecolor=black,bookmarksopen=true,unicode=true]{hyperref}
\usepackage[norefs,nocites]{refcheck}
\usepackage{amsmath}
\usepackage{amsfonts}
\usepackage{amssymb}
\usepackage[T1]{fontenc}
\usepackage{graphicx}
\usepackage{xcolor}
\usepackage{colortbl}
\usepackage{float}
\usepackage{framed}
\usepackage{pgf,tikz,pgfplots}
\usepackage{tkz-base}
\usepackage{multicol}
\usepackage{tabularx}
\usepackage{tkz-fct}%
\providecommand{\U}[1]{\protect\rule{.1in}{.1in}}
\tikzset{>=Triangle}
\newtheorem{theorem}{Theorem}[section]

\newtheorem{remark}[theorem]{Remark}

\newtheorem{lemma}[theorem]{Lemma}
\newtheorem{problem}[theorem]{Problem}
\newtheorem{definition}[theorem]{Definition}
\numberwithin{equation}{section}
\pgfplotsset{compat=1.17}
\begin{document}
\title[On a theorem due to Murray]{On a theorem due to Murray}
\author[Barbosa]{Anderson Barbosa}
\address[A. Barbosa]{Departamento de Matem\'atica \\
Universidade Federal da Para\'iba \\
58.051-900 Jo\~ao Pessoa, Brazil.}
\email{afsb@academico.ufpb.br}
\author[A. Raposo Jr.]{Anselmo Raposo Jr.}
\address[A. Raposo Jr.]{Departamento de Matem\'{a}tica \\
Universidade Federal do Maranh\~{a}o \\
65085-580 - S\~{a}o Lu\'{\i}s, Brazil.}
\email{anselmo.junior@ufma.br}
\author[G. Ribeiro]{Geivison Ribeiro}
\address[G. Ribeiro -- \textit{Corresponding author}]{Departamento de Matem\'{a}tica \\
Universidade Federal da Para\'{\i}ba \\
58.051-900 - Jo\~{a}o Pessoa, Brazil.}
\email{geivison.ribeiro@academico.ufpb.br}
\thanks{$^{\ast}$Corresponding author.}
\thanks{The third author was supported by Grant 2022/1962 from the Para\'{\i}ba State
Research Foundation (FAPESQ) }

\begin{abstract}
In this paper, we introduce the notions of $\alpha$-quasicomplemented and
totally $\alpha$-quasicomplemented subspaces and we established some results
under these contexts. We show, for example, that if $X$ is a separable or
reflexive Banach space and $Y$ is a closed infinite codimensional subspace of
$X$, then $Y$ is totally$\mathit{\ }\alpha$-quasicomplemented if, and only if,
$\alpha<\aleph_{0}$ $\left(  \text{this is an analogue of the theorem of
Murray-Mackey and Lindenstrauss}\right)  $. We also show that if $H$ is a
Hilbert space and $Y,W\subset H$ are closed subspaces of $H$ such that $W$ is
orthogonal to $Y$ and $\operatorname{codim}\left(  Y+W\right)  =\infty$, then
$Y$ has a quasicomplement $Z$ containing $W$ with $\dim Z/W=\infty$. Other
results in the different contexts are also included. Such results establish a
connection between the theory of quasicomplemented subspaces and $\left(
\alpha,\beta\right)  $-spaceability.

\end{abstract}
\maketitle

\section{Introduction and background}

If $X$ is a topological vector space, we say that a closed subspace $Y$ of $X$
is \textbf{quasicomplemented} in $X$ if there is a closed subspace $Z$ in $X$
such that\ $Y\cap Z=\left\{  0\right\}  $ and $Y+Z$ is dense in $X$. In
addition, if $Y+Z\neq X$, then we say that $Y$ is proper quasicomplemented in
$X$. The subspaces $Y$ and $Z$ above are said be \textbf{quasicomplements} (if
$Y+Z$ is not closed, then we say that $Y$ and $Z$ are \textbf{proper
quasicomplements}). This concept was coined by Murray \cite{MURRAY}, who
proved that every closed subspace of a separable and reflexive space is
quasicomplemented. Removing the reflexivity hypothesis, Mackey \cite{MACKEY}
showed that the same statement is true in any separable Banach space. More
precisely, if $X$ is a separable Banach space, then every closed subspace of
$X$ is quasicomplemented in $X$. The last result mentioned above is known as
the Murray-Mackey theorem for the separable case. A few years later,
Lindenstrauss \cite{Lindenstrauss} extended Murray's results to show that
every subspace of $X$ is quasicomplemented if $X$ is only reflexive. Still in
\cite{Lindenstrauss}, Lindenstrauss asks whether or not $c_{0}$ is proper
quasicomplemented in $\ell_{\infty}$ and in \cite{L}, he proves that
$c_{0}\left(  \Gamma\right)  $ is not proper quasicomplemented in
$\ell_{\infty}\left(  \Gamma\right)  $ if $\Gamma$ is uncountable. In
\cite[Theorem 1.7]{Rosenthal} Rosenthal gave a positive answer to this
question showing that $c_{0}$ is proper quasicomplemented in $\ell_{\infty}$.
Also in \cite{Rosenthal}, a nonseparable extension of Murray-Mackey theorem
was provided by Rosenthal, who proved that a closed subspace $F$ of a Banach
space $X$ is quasicomplemented whenever the topological dual of $F$, denoted
by $F^{\ast}$, is $w^{\ast}$-separable and the annihilator of $F$ contains a
reflexive subspace.

The concepts defined and explored in this paper, among other things, end up
establishing a link between the theory of quasicomplemented subspaces and the
theory of lineability within the scope of spaceability. The term lineability
was coined by V.I. Gurariy \cite{AGSS,JBSS,aron} and since then this topic has
been deeply investigated. Let $X$ be a vector space, $M$ be a nonempty subset
of $X$ and $\alpha,\beta$ be cardinal numbers such that $\alpha\leq\beta
\leq\dim X$, where $\dim X$ denotes the cardinality of a Hamel basis of $X$.
We say that $M$ is $\alpha$-\textbf{lineable} if there is an $\alpha
$-dimensional subspace $W$ of $X$ such that%
\[
W\subset M\cup\left\{  0\right\}  \text{.}%
\]
When $X$ is endowed with a topology and the $\alpha$-dimensional subspace $W$
can be chosen closed we say that $M$ is $\alpha$-\textbf{spaceable}. Also in
the context in which $X$ is endowed with a topology, we say that $M$ is
$\left(  \alpha,\beta\right)  $-\textbf{spaceable} if it is $\alpha$-lineable
and for each $\alpha$-dimensional subspace $W_{\alpha}\subset M\cup\left\{
0\right\}  $ there is a $\beta$-dimensional closed subspace $W_{\beta}$ such
that
\[
W_{\alpha}\subset W_{\beta}\subset M\cup\left\{  0\right\}  \text{.}%
\]
A detailed account of $\left(  \alpha,\beta\right)  $-spaceability can be
found in \cite{Diogo/Anselmo, Diogo, Pilar, FPT, Pellegrino}.

\begin{definition}
Let $X$ be a topological vector space and let $\alpha$ be a cardinal number.
We say that a closed subspace $Y$ of $X$ is $\alpha$%
-\textbf{quasicomplemented} in $X$ if there is\textbf{ }a quasicomplement $Z$
to $Y$ in $X$ containing an $\alpha$-dimensional subspace $F$ such that $\dim
Z/F=\infty$.
\end{definition}

\begin{definition}
Let $X$ be a topological vector space and let $\alpha$ be a cardinal number.
We say that a closed subspace $Y$ of $X$ is \textbf{totally} $\alpha
$-\textbf{quasicomplemented} in $X$ if $Y$ is quasicomplemented in $X$ and for
each $\alpha$-dimensional subspace $F$ of $X$ with $F\cap Y=\left\{
0\right\}  $, there exists a quasicomplement $Z$ of $Y$ containing $F$ such
that $\dim Z/F=\infty$.
\end{definition}

In other words, we say that $Y$ is totally\textit{ }$\alpha$-quasicomplemented
in $X$ if for each $\alpha$-dimensional subspace $F$ of $X$ with $F\cap
Y=\left\{  0\right\}  $, there exists a closed subspace $Z$ of $X$ such that%
\[
F\subset Z\text{, \ \ \ \ }Y\cap Z=\left\{  0\right\}  \text{, \ \ \ \ }\dim
Z/F=\infty\text{ \ \ \ \ and \ \ \ \ }Y+Z\text{ is dense in }X\text{.}%
\]
In addition, if $Y+Z$ not is closed, we say that $Y$ is \textbf{properly
totally} $\alpha$-\textbf{quasicomplemented}.

Obviously, each totally $\alpha$-quasicomplemented subspace is also $\alpha
$-quasicomplemented. However, this notions are not equivalent. Indeed, let $X$
be an infinite dimensional Banach space with Schauder basis $\left(
x_{n}\right)  _{n=1}^{\infty}$ and let $Y:=\mathbb{K}x_{1}$ and $Z:=\overline
{\operatorname*{span}}\left\{  x_{n}\right\}  _{n=2}^{\infty}$ It is well
known that, in this circumstances, $X=Y+Z$ and $Y\cap Z=\left\{  0\right\}  $.
In this case, since $\dim Z=\dim X=\infty$ we can take $F:=\overline
{\operatorname*{span}}\left\{  x_{2n}\right\}  _{n=1}^{\infty}$ and obtain
that $\dim Z/F=\mathfrak{c}$ where $\mathfrak{c}$ denotes the cardinality of
the set $\mathbb{R}$. This show that $Y$ is $\mathfrak{c}$-quasicomplemented
in $X$. However, F\'{a}varo et al. ensures in \cite[Theorem 2.1]{Raposo} that,
in this case, the set $X\setminus Y$ is not $\left(  \alpha,\mathfrak{c}%
\right)  $-spaceable for $\aleph_{0}\leq\alpha\leq\mathfrak{c}$, where
$\aleph_{0}$ denotes the cardinality of the set $\mathbb{N}$ of positive
integers. This means that, if $\aleph_{0}\leq\alpha\leq\mathfrak{c}$, then $Y$
is not totally $\alpha$-quasicomplemented.

In this paper, if $X$ is a vector space and $Y$ is a subspace of $X$, the
\textbf{codimension} of $Y$, symbolically denoted by $\operatorname{codim}Y$,
will indicate the dimension of the quotient space $X/Y$.

This paper consists of three sections, in addition to this Introduction. In
the next section, we establish characterization results for $\alpha
$-quasicomplementation, which offer a form of extension to the theorems of
Murray, Mackay, and Lindenstrauss in the context of quasicomplemented
subspaces. In Section 3, we further establish characterization results for
$\alpha$-quasicomplementation, but now in the general context of topological
vector spaces. In the final section, our attention is directed towards some
observations within the scope of spaceability for complements of vector subspaces.

\section{Totally $\alpha$-quasicomplemented subspaces in Banach spaces}

\begin{problem}
\label{Prob1} Is every infinite dimensional quasicomplemented subspace $Y$ of
a Banach space $X$ totally $\alpha$-quasicomplemented for every $0<\alpha
<\aleph_{0}$?
\end{problem}

The first two results of this section, supported by the results of Murray,
Mackay, and Lindenstrauss for quasicomplemented subspaces and \cite[Theorem
2.1]{Raposo}, provide a kind of partial (\textit{positive}) solution to
Problem \ref{Prob1}, as well as establishing a connection between $\left(
\alpha,\beta\right)  $-spaceability and the theory of quasicomplements in the
context of Banach spaces.

For the sake of simplicity, we combine the results of Murray, Mackay, and
Lindenstrauss as follows:

\begin{theorem}
[Lindenstrauss--Mackey--Murray]\label{Mackey}If $X$ is a separable or
reflexive Banach space, then every closed subspace of $X$ is quasicomplemented.
\end{theorem}

The spaceability result used in our approach is stated as follows:

\begin{theorem}
\label{TeoFPRR}$($See \cite[Theorem 2.1]{Raposo}$)$ Let $\alpha\geq\aleph_{0}$
and $X$ be an $F$-space. Let $A$, $B$ be subsets of $X$ such that $A$ is
$\alpha$-lineable and $B$ is $1$-lineable. If $A\cap B=\varnothing$ and $A$ is
stronger than $B$ $($i.e. $A+B=\left\{  a+b:a\in A\text{ and }B\in B\right\}
\subset A)$, then $A$ is not $\left(  \alpha,\beta\right)  $-spaceable,
regardless of the cardinal number $\beta$.
\end{theorem}

An important ingredient in the proof of the Theorem \ref{Theorem MM} lies in
this\emph{ }simple and useful remark:

\begin{remark}
\label{Lemma para soma} Let $X$ be a Hausdorff topological vector space. If
$Y$ and $F$ are subspaces of $X$ such that $Y$ is closed in $X$ and $F$ is
finite dimensional, then $Y+F$ is closed in $X$. In fact, let $Q\colon
X\longrightarrow X/Y$ be the quotient map of $X$ onto $X/Y$. Since $F$ is
finite dimensional the subspace $Q\left(  F\right)  $ has finite dimension.
This implies that $Q\left(  F\right)  $ is closed in $X/Y$. Since
$Y+F=Q^{-1}\left(  Q\left(  F\right)  \right)  $ and $Q$ is continuous, we get
that $Y+F$ is closed in $X$.
\end{remark}

We are now able to state the first result of this section.

\begin{theorem}
\label{Theorem MM}If $X$ is a separable or reflexive Banach space, then every
infinite codimensional closed subspace of $X$ is totally$\mathit{\ }\alpha
$-quasicomplemented if, and only if, $\alpha<\aleph_{0}$.
\end{theorem}

\begin{proof}
Fix $\alpha<\aleph_{0}$ and let $Y$ be an infinite codimensional closed
subspace of $X$. Let $F$ be an $\alpha$-dimensional subspace of $X$ such that
$F\cap Y=\left\{  0\right\}  $. Since $Y+F$ is closed in $X$ (see Remark
\ref{Lemma para soma}) and $\operatorname{codim}\left(  Y+F\right)
=\operatorname{codim}Y=\infty$, we can invoke Theorem \ref{Mackey} to obtain a
quasicomplement $Z$ to $Y+F$ in $X$. Taking%
\[
W:=F+Z\text{,}%
\]
and using Remark \ref{Lemma para soma} again, we\ can infer that $W$ is
closed. Moreover, it is plain that
\[
W\cap Y=\left\{  0\right\}  \text{, \ }F\subset W\text{, \ }\dim W/F=\dim
Z=\infty\text{ \ and \ }Y+W=\left(  Y+F\right)  +Z\text{ is dense in
}X\text{.}%
\]
Therefore, $Y$ is totally $\alpha$-quasicomplemented. The converse is an
immediate consequence of Theorem \ref{TeoFPRR}.
\end{proof}

\begin{theorem}
Let $X$ be a Banach space. If $Y$ is an infinite codimensional closed subspace
of $X$ such that $X/Y$ is separable, then $Y$ is totally$\mathit{\ }\alpha
$-quasicomplemented if,\ and only if, $\alpha<\aleph_{0}$.
\end{theorem}

\begin{proof}
Fix $\alpha<\aleph_{0}$ and let $F$ be an $\alpha$-dimensional subspace of $X$
such that
\[
F\cap Y=\left\{  0\right\}  \text{.}%
\]
Let $D=\left\{  d_{n}:n\in\mathbb{N}\right\}  $ be a dense subset of $X/Y$.
Without loss of generality we can consider that $d_{n}\neq0$ for every
$n\in\mathbb{N}$. Since the quotient map $Q\colon X\longrightarrow X/Y$ is
surjective, for each $n\in\mathbb{N}$ there is $e_{n}\in X\setminus Y$ such
that $Q\left(  e_{n}\right)  =d_{n}$. Let%
\[
W_{1}:=\overline{\operatorname*{span}}(\left\{  e_{n}:n\in\mathbb{N}\right\}
\cup\mathcal{B})\text{,}%
\]
where $\mathcal{B}$ is a Hamel basis to $F$. We claim that
\[
Y+W_{1}\text{ is dense in }X\text{.}%
\]
In fact, fix $x\in X$. Since $D$ is dense in $X/Y$ there is a sequence
$\left(  d_{n_{k}}\right)  _{k=1}^{\infty}$ in $D$ such that
\[
d_{n_{k}}-Q\left(  x\right)  \overset{k\rightarrow\infty}{\longrightarrow
}0\text{.}%
\]
It follows from open mapping theorem, passing to a subsequence of $\left(
d_{n_{k}}\right)  _{k=1}^{\infty}$ if necessary, that there is a sequence
$z_{_{k}}\in X$ such that%
\[
Q\left(  z_{k}\right)  =d_{n_{k}}-Q\left(  x\right)  \text{ \ \ \ \ and
\ \ \ \ }z_{k}\overset{k\rightarrow\infty}{\longrightarrow}0\text{.}%
\]
Thus the sequence $x_{k}:=z_{k}+x\in X$ is such that
\[
Q\left(  x_{k}\right)  =d_{n_{k}}\text{ \ \ \ \ and \ \ \ \ }x_{k}%
=z_{k}+x\overset{k\rightarrow\infty}{\longrightarrow}x\text{.}%
\]
In particular,%
\[
Q\left(  x_{k}\right)  =d_{n_{k}}=Q\left(  e_{n_{k}}\right)
\]
and, therefore,
\[
x_{k}-e_{n_{k}}\in Y\text{ \ \ \ \ and \ \ \ \ }x_{k}=\left(  x_{k}-e_{n_{k}%
}\right)  +e_{n_{k}}\in Y+W_{1}\text{.}%
\]
If $W_{1}\cap Y=\left\{  0\right\}  $ the result is done since $F\subset
W_{1}$ and $\dim W_{1}/F=\infty$. Otherwise, let
\[
N_{1}:=W_{1}\cap Y\text{.}%
\]
Note that $N_{1}$ is closed in $W_{1}$ and $N_{1}\cap F=\left\{  0\right\}  $.
Since $W_{1}$ is separable, we can invoke the Theorem \ref{Theorem MM} to
obtain a quasicomplement $W_{2}$ to $N_{1}$ in $W_{1}$ such that
\[
F\subset W_{2}\text{.}%
\]
Since
\[
W_{2}\cap Y=\left(  W_{2}\cap W_{1}\right)  \cap Y=W_{2}\cap\left(  W_{1}\cap
Y\right)  =W_{2}\cap N_{1}=\left\{  0\right\}  \text{,}%
\]
all we need to do is prove that $Y+W_{2}$ is dense in $X$ and $\dim
W_{2}/F=\infty$. \ First, let us show that $Y+W_{2}$ is dense in $X$. In fact,
pick $v\in X\setminus\left\{  0\right\}  $. For $\varepsilon>0$ there are
$w\in W_{1}$ and $y\in Y$ such that $\left\Vert \left(  w+y\right)
-v\right\Vert <\varepsilon/2$. Furthermore, since $N_{1}+W_{2}$ is dense
$W_{1}$, there are $z\in W_{2}$ and $u\in N_{1}$ such that $\left\Vert \left(
z+u\right)  -w\right\Vert <\varepsilon/2$. Thus
\[
\left\Vert z+u+y-v\right\Vert \leq\left\Vert \left(  z+u\right)  -w\right\Vert
+\left\Vert \left(  w+y\right)  -v\right\Vert <\varepsilon\text{.}%
\]
Since $u+y\in Y$ we conclude that $Y+W_{2}$ is dense in $X$. Now, we will show
that $\dim W_{2}/F=\infty$. Indeed, since $Y+W_{2}$ is dense in $X$, if $\dim
W_{2}<\infty$ then $\dim X/Y<\infty$ because the subspace $W_{2}$ is
isomorphic to $X/Y$. But $\dim X/Y=\infty$. Therefore, $\dim W_{2}=\infty$ and
consequently $\dim W_{2}/F=\infty$. The converse, as before, is an immediate
consequence of Theorem \ref{TeoFPRR}.
\end{proof}

\bigskip

The subsequent result presented in this section represents a nuanced extension
of the projection theorem found in \cite[Theorem 3.3-4]{Kreyszig}.
Importantly, this result offers conditions for a closed subspace to be
$\alpha$-quasicomplemented, with '$\alpha$' inclusively representing an
infinite cardinal.

\begin{theorem}
Let $H$ be a Hilbert space and $Y\subset H$ be a closed subspace of $H$. If
$W$ is a subspace of $H$ orthogonal to $Y$ and
\[
\operatorname{codim}\left(  Y+W\right)  =\infty\text{,}%
\]
then $Y$ has a quasicomplement $Z$ $($not necessarily identical to the
orthogonal complement of $Y)$ containing $W$ and so that $\dim Z/W=\infty$. In
other words, $Y$ is $\dim W$-quasicomplemented in $H$.
\end{theorem}

\begin{proof}
Let $\left\langle x,y\right\rangle $ be the inner product in $H$ such that
\[
\left\Vert x\right\Vert _{H}:=\left\langle x,x\right\rangle ^{\frac{1}{2}%
}\text{, for every }x\in H\text{.}%
\]
Since $W$ is orthogonal to $Y$ we get
\[
\left\Vert x+y\right\Vert _{H}^{2}=\left\Vert x\right\Vert _{H}^{2}+\left\Vert
y\right\Vert _{H}^{2}\geq\left\Vert x\right\Vert _{H}^{2}\text{ for every
}x\in W\text{ and }y\in Y\text{.}%
\]
By \cite[Theorem 1]{Kober}, $Y+W$ is closed in $H$. According to \cite[Theorem
3.3-4]{Kreyszig}, there exists a closed subspace $F$ in $H$ such that%
\[
F\text{ is orthogonal to }Y+W\text{, \ \ \ \ }\left(  Y+W\right)  \cap
F=\left\{  0\right\}  \text{ \ \ \ \ and \ \ \ \ }\left(  Y+W\right)
+F=H\text{.}%
\]
Since $\operatorname{codim}\left(  Y+W\right)  =\infty$, the subspace $F$ is
infinite dimensional. Moreover, since $F$ is orthogonal to $Y+W$ and $W$ is
orthogonal to $Y$ we easily conclude that $Y$ is orthogonal to $W+F$. Due to
the fact that $W$ is orthogonal to $F$ we can infer that again by
\cite[Theorem 1]{Kober} that $W+F$ is closed. Since
\[
Y\cap\left(  W+F\right)  =\left\{  0\right\}  \ \ \ \ \ \text{and
\ \ \ \ }Y+\left(  W+F\right)  =H
\]
and, furthermore,%
\[
W\cap F=\left\{  0\right\}  \text{,}%
\]
the result is done, if we take $Z:=W+F$.
\end{proof}

\bigskip

\section{Totally $\alpha$-quasicomplemented subspaces in Hausdorff topological
vector spaces}

In this section, we provide a result that also offers a partial and positive
solution to Problem \ref{Prob1}, now in a more general context.

\begin{lemma}
\label{l2.2}Let $X$ be a Hausdorff topological vector space. If $X=Y\oplus F$
and $F$ is finite dimensional then $X$ is isomorphic to $Y\times F$.
\end{lemma}

\begin{proof}
Let $T\colon Y\times F\longrightarrow X$ be a map given by $T(u,v)=u+v$. Since
$X=Y\oplus F$ the map $T$ is an algebraic isomorphism. Furthermore, $T$ is
continuous since it is the restriction of the addition operation in $X$. Let
us show that the inverse of $T$ is also continuous. First note that the
quotient map restriction $Q|_{F}\colon F\rightarrow X/Y$ is a topological
isomorphism (see \cite[Theorem 1.21]{Rudin}). Moreover, if $(u_{\lambda
}+v_{\lambda})_{\lambda\in\Lambda}$ is a net in $X$ converging to $u+v$, then
we have%
\[
Q(v_{\lambda})=Q(u_{\lambda}+v_{\lambda})\overset{\lambda}{\longrightarrow
}Q(u+v)=Q(v)\text{.}%
\]
Since $(Q|_{F})^{-1}$ is continuous, we conclude that $v_{\lambda}%
\overset{\lambda}{\longrightarrow}v$. Thus%
\[
u_{\lambda}=\left(  u_{\lambda}+v_{\lambda}\right)  -v_{\lambda}%
\overset{\lambda}{\longrightarrow}\left(  u+v\right)  -v=u\text{.}%
\]
Therefore, the result is done.
\end{proof}

\begin{theorem}
\label{X}Let $X$ be a\emph{ }Hausdorff\emph{ }locally convex space and
$Y\subset X$ be a subspace isomorphic to $\mathbb{K}^{I}$ for some set $I$.
Then $Y$ is closed in $X$, and furthermore, $Y\text{ is }$totally $\alpha
$-quasicomplemented in $X$ for all $\alpha<\min\{\aleph_{0},\dim X/Y\}$.
\end{theorem}

\begin{proof}
Let $F$ be a subspace of $X$ such that%
\[
\dim F<\min\{\aleph_{0},\dim X/Y)\}\text{ \ \ \ \ and \ \ \ \ }F\cap
Y=\{0\}\text{.}%
\]
Let $I$ be a set such that $Y$ is isomorphic to $\mathbb{K}^{I}$. Since
$\mathbb{K}^{I}$ is complete and $X$ is Hausdorff, we conclude that $Y$ is
closed in $X$. By the remark \ref{Lemma para soma}, the vector subspace
$W:=F+Y$ is closed in $X$. We claim that the subspace $W$ is isomorphic to
$\mathbb{K}^{J}$, where
\[
\operatorname*{card}(J)=\left\{
\begin{array}
[c]{ll}%
\operatorname*{card}(I)\text{,} & \text{if }\operatorname*{card}%
(I)=\infty\text{,}\\
\operatorname*{card}(I)+\dim F\text{,} & \text{if }\operatorname*{card}%
(I)<\infty\text{.}%
\end{array}
\right.
\]
In fact, write $J=I\cup I^{\prime}$, where $I\cap I^{\prime}=\varnothing$ and
$\dim F=\operatorname{card}(I^{\prime})$. By Lemma \ref{l2.2}, we have
$\mathbb{K}^{I}\oplus\mathbb{K}^{I^{\prime}}\simeq\mathbb{K}^{I}%
\times\mathbb{K}^{I^{\prime}}$ and $Y\times F\simeq W$. Thus
\[
\mathbb{K}^{J}=\mathbb{K}^{I}\oplus\mathbb{K}^{I^{\prime}}\simeq\mathbb{K}%
^{I}\times\mathbb{K}^{I^{\prime}}\simeq Y\times F\simeq W\text{.}%
\]
Now, let $\phi=\left(  \phi_{j}\right)  _{j\in J}\colon W\longrightarrow
\mathbb{K}^{J}$ be an isomorphism, where $\phi_{j}\in W^{\ast}$ for every
$j\in J$. By the Hahn-Banach theorem, there is $\varphi_{j}\in X^{\ast}$ such
that%
\[
\varphi_{j}\left(  w\right)  =\phi_{j}\left(  w\right)  \text{ for each }w\in
W\text{.}%
\]
The map $\varphi=\left(  \varphi_{j}\right)  _{j\in J}\colon X\longrightarrow
\mathbb{K}^{J}$ is a continuous linear surjection whose restriction to $W$ is
$\phi$. Since the composition $\phi^{-1}\circ\varphi\colon X\longrightarrow W$
is a continuous projection onto $W$, there is a closed subspace $G$ of $X$
such that $G\cap W=\left\{  0\right\}  $ and $G+W=X$. The subspace $Z:=F+G$ is
a quasicomplement of $Y$ containing $F$. More precisely, $F\subset Z$ and
$Z\cap Y=\left\{  0\right\}  $ and $Y+Z=X$. Since $\dim Z/F=\infty$, the
result is done.
\end{proof}

\section{Final remarks in the context of spaceability}

A classical result due to Wilansky and Kalton says that if $Y$ is a closed
subspace of a Fr\'{e}chet space $X$, then $X\setminus Y$ is spaceable if, and
only if, $Y$ is infinite codimensional. In fact, Wilansky \cite[p.12]%
{Wilansky} proved this result for Banach spaces and Kalton noticed that the
same proof works for Fr\'{e}chet spaces (complete metrizable locally convex
vector spaces). This result appears in \cite{Kitson} in the following way:

\begin{theorem}
$($\cite[Theorem 2.2]{Kitson}$)$ If $Y$ is a closed subspace of a Fr\'{e}chet
space $X$, then $X\setminus Y$ is spaceable if, and only if, $Y$ has infinite codimension.
\end{theorem}

Theorem \ref{X} shows, in particular, that if $Y\subset X$ is a subspace
isomorphic to $\mathbb{K}^{\mathbb{N}}$ then $X\setminus Y$ is $\left(
n,\operatorname{codim}Y\right)  $-spaceable for each $n\in\mathbb{N}$.

In \cite{GA} it was verified that regardless of the subspace $Y$ of a space
$X$, the $\left(  \alpha,\operatorname{codim}Y\right)  $-lineability of the
set $X\setminus Y$ always holds for any $\alpha<\operatorname{codim}Y$.
Evidently, this result is the best possible in the context of $\left(
\alpha,\beta\right)  $-lineability of the complement of vector subspaces.
However, it is already known that for certain types of topological vector
spaces, the optimal result obtained in \cite{GA} does not admit an analogue in
the context of $\left(  \alpha,\beta\right)  $-spaceability, as is the case of
\cite[Theorem 2.1]{Raposo}, since we can only consider $\alpha<\aleph_{0}$ in
the scope of $F$-spaces. Despite the non-validity of $\left(  \alpha
,\beta\right)  $-spaceability for $\alpha\geq\aleph_{0}$ in general, many
results obtained for specific types of subspaces $Y$ yield $\left(
n,\operatorname{codim}Y\right)  $-spaceability for all $n\in\mathbb{N}$, as is
the case of Theorem \ref{X} (where $Y$ is closed), the result obtained in
\cite{GAAG} (where $Y$ is dense and $X$ is separable) and several applications
of the main results of \cite{GAAG2} (where there are examples where $X$ is not
separable). Even with strong evidences that the spaceability of $X\setminus Y$
implies the $(n,\operatorname{codim}Y)$-spaceability of $X\setminus Y$ for all
$n\in\mathbb{N}$ (at least for Fr\'{e}chet spaces), we do not know what
minimum hypotheses must apply to $Y$ and $X$ for such a result to be true.

Given this, we have the following problem:

\begin{problem}
If $Y$ is a closed subspace of a Fr\'{e}chet space $X$ with $\dim Y$,
$\operatorname{codim}Y=\infty$, and $X\setminus Y$ is spaceable, then is it
also $\operatorname{codim}Y$-spaceable?
\end{problem}


\begin{thebibliography}{99}                                                                                               %


\bibitem {GA}G. Ara\'{u}jo, A. Barbosa, \textit{A general lineability
criterion for complements of vector spaces}, Rev. Real Acad. Cienc. Exactas
Fis. Nat. Ser. A Mat. \textbf{118} (2024), no. 1, Paper No. 5.

\bibitem {GAAG}G. Ara\'{u}jo, A. Barbosa, A. Raposo Jr. and G. Ribeiro,
\textit{On the spaceability of the set of functions in the Lebesgue space
}$L_{p}$\textit{ which are not in} $L_{q}$, Bull. Braz. Math. Soc. (N.S.)
\textbf{54} (2023), no. 3, Paper No. 44.

\bibitem {GAAG2}G. Ara\'{u}jo, A. Barbosa, A. Raposo Jr. and G. Ribeiro,
$(\alpha,\beta)$-\textit{spaceability and applications}, preprint (2023), arXiv:2306.01561.

\bibitem {aron}R.M. Aron, L. Bernal-Gonz\'{a}lez, D. Pellegrino and J.B.
Seoane-Sep\'{u}lveda, Lineability: \textit{the search for linearity in
mathematics, Monographs and Research Notes in Mathematics}. CRC Press, Boca
Raton, FL, 2016. xix+308 pp.

\bibitem {AGSS}R.M. Aron, V.I. Gurariy and J.B. Seoane-Sep\'{u}lveda,
\textit{Lineability and spaceability of sets of functions on }$\mathbb{R}$,
Proc. Amer. Math. Soc. \textbf{133} (2004), 795-803.

\bibitem {Diogo/Anselmo}D. Diniz and A. Raposo Jr, \textit{A note on the
geometry of certain classes of linear operators}. Bull. Braz. Math. Soc.
(N.S.) \textbf{52} (2021), 1073-1080.

\bibitem {Diogo}D. Diniz, V.V. F\'{a}varo, D. Pellegrino and A. Raposo Jr,
\textit{Spaceability of the sets of surjective and injective operators between
sequence spaces}. Rev. Real Acad. Cienc. Exactas Fis. Nat. Ser. A Mat.
\textbf{114} (2020), no. 4, Paper No. 194.

\bibitem {Raposo}V.V. F\'{a}varo, D. Pellegrino, A. Raposo J\'{u}nior, G.S.
Ribeiro, \textit{General criteria for a stronger notion of lineability}, Proc.
Amer. Math. Soc. \textbf{152} (2024), no. 3, 941-954.

\bibitem {Pilar}V.V. F\'{a}varo, D. Pellegrino and P. Rueda, \textit{On the
size of the set of unbounded multilinear operators between Banach spaces}.
Linear Algebra Appl. \textbf{606} (2020), 144-158.

\bibitem {FPT}V.V. F\'{a}varo, D. Pellegrino and D. Tom\'{a}z,
\textit{Lineability and spaceability: a new approach}, Bull. Braz. Math. Soc.
New Ser. \textbf{51} (2019), 27-46.

\bibitem {Kitson}D. Kitson and R.M. Timoney, \textit{Operator ranges and
spaceability}. J. Math. Anal. Appl. \textbf{378} (2011), 680-686.

\bibitem {Kober}H. Kober, \textit{A theorem on Banach spaces}, Compos. Math.
\textbf{7} (1940), 135-140.

\bibitem {Kreyszig}E. Kreyszig, \textit{Introductory functional analysis with
applications}. Wiley Classics Library. John Wiley \& Sons, Inc., New York,
1989. xvi+688 pp.

\bibitem {Lindenstrauss}J. Lindenstrauss, \textit{On a theorem of Murray and
Mackey}. An. Acad. Brasil. Ci. \textbf{39} (1967), 1-6.

\bibitem {L}J. Lindenstrauss, \textit{On subspaces of Banach spaces without
quasicomplements}. Israel Journal of Math. \textbf{6} (1968), 36-38.

\bibitem {MACKEY}G. Mackey, \textit{Note on a theorem of Murray}. Bull. Am.
Math. SOL. \textbf{52} (1946), 322-325.

\bibitem {MURRAY}F.J Murray, \textit{Quasi-complements and closed projections
in reflexive Banach spaces}. Trans. Am. Math. Sot. \textbf{58} (1945), 77-95.

\bibitem {Pellegrino}D. Pellegrino and A. Raposo Jr, \textit{Pointwise
lineability in sequence spaces}. Indag. Math. (N.S.) \textbf{32} (2021), 536-546.

\bibitem {Rosenthal}H.P. Rosenthal, \textit{On quasicomplemented subspaces of
Banach spaces}. Proc. Nat. Acad. Sci. U.S.A. \textbf{59} (1968), 361-364.

\bibitem {Rudin}W. Rudin, \textit{Functional analysis}. Second edition.
International Series in Pure and Applied Mathematics. McGraw-Hill, Inc., New
York, 1991. xviii+424 pp.

\bibitem {JBSS}J.B. Seoane-Sep\'{u}lveda, \textit{Chaos and lineability of
pathological phenomena in analysis}, Thesis (Ph.D.)--Kent State University.
2006. \textbf{139} pp.

\bibitem {Wilansky}A. Wilansky, \textit{Semi-Fredholm maps of FK spaces}.
Math. Z. \textbf{144} (1975), 9-12.
\end{thebibliography}
\end{document}